\documentclass[11pt]{amsart}

\usepackage{amsmath}
\usepackage{amssymb}
\usepackage{cite}

\newtheorem{theorem}{Theorem}[section]
\newtheorem{corollary}[theorem]{Corollary}

\theoremstyle{definition}
\newtheorem{definition}[theorem]{Definition}

\newtheorem{remark}[theorem]{Remark}
\newtheorem*{problem}{Problem $\mathbf{(H_\tau^n)}$}

\numberwithin{equation}{section}

\AtBeginDocument{{\noindent\small
This is a preprint of a paper whose final and definite form will be published 
in the International Journal \emph{Pure and Applied Functional Analysis}, 
ISSN 2189-3756 (print), ISSN 2189-3764 (online). 
Submitted 01-Dec-2015; revised 12-Mar-2016; accepted for publication 13-Mar-2016.}
\vspace{9mm}}

\title[Variational problems of Herglotz type with time delay]{Higher-order
variational problems\\ of Herglotz type with time delay}

\author[S. P. S. Santos, N. Martins and D. F. M. Torres]{Sim\~{a}o P. S. Santos,
Nat\'{a}lia Martins and Delfim F. M. Torres}

\address[Sim\~{a}o P. S. Santos, Nat\'{a}lia Martins and Delfim F. M. Torres]{Center
for Research and Development in Mathematics and Applications (CIDMA),
Department of Mathematics, University of Aveiro, 3810-193 Aveiro, Portugal}
\email{{\tt spsantos@ua.pt, natalia@ua.pt, delfim@ua.pt}}



\begin{document}

\begin{abstract}
We study, using an optimal control point of view, higher-order variational
problems of Herglotz type with time delay. Main results are higher-order
Euler--Lagrange and DuBois--Reymond necessary optimality conditions as well
as a higher-order Noether type theorem for delayed variational
problems of Herglotz type.
\end{abstract}

\maketitle


\section{Introduction}

The generalized variational problem proposed by Herglotz in 1930 \cite{Herglotz1930}
can be formulated as follows: determine the trajectories $x\in C^1([a,b];\mathbb{R}^m)$
and the function $z\in C^1([a,b];\mathbb{R})$ such that
\begin{equation}
\label{PH}
\tag{$H^1$}
\begin{gathered}
z(b)\longrightarrow \textrm{extr},\\
\text{where the pair } (x(\cdot),z(\cdot)) \text{ satisfies }\\
\dot{z}(t)=L(t,x(t),\dot{x}(t),z(t)), \quad t \in [a,b],\\
\text{subject to } x(a)=\alpha, \quad z(a)=\gamma,
\end{gathered}
\end{equation}
for some $\alpha \in \mathbb{R}^m$ and $\gamma \in \mathbb{R}$, and
where the Lagrangian $L$ is assumed to satisfy the following assumptions:
\begin{itemize}
\item[i.] $L \in C^1([a,b]\times\mathbb{R}^{2m+1}; \mathbb{R})$;

\item[ii.] the functions $\displaystyle t \mapsto \frac{\partial L}{\partial x}\left(t,
x(t), \dot{x}(t), z(t)\right)$, $\displaystyle t
\mapsto \frac{\partial L}{\partial \dot{x}}\left(t,
x(t), \dot{x}(t), z(t)\right)$  and $\displaystyle t \mapsto
\frac{\partial L}{\partial z}\left(t, x(t), \dot{x}(t), z(t)\right)$
are differentiable for any admissible pair $(x(\cdot),z(\cdot))$.
\end{itemize}
Observe that if the Lagrangian $L$ does not depend on the variable $z$, 
then we get the classical problem of the calculus of variations.
An advantage of formulation \eqref{PH} is the ability to provide a variational description
of non-conservative processes, even when the Lagrangian $L$ is autonomous, something
that is not possible with the classical variational problem. For important
applications in thermodynamics see \cite{Georgieva2002}.

The variational problem of Herglotz attracted the interest of the mathematical
community only in 1996, with the publications \cite{Guenther1996SIAM,Guenther1996}.
Since then, several authors investigated such variational problems.
The following generalizations  of well-known classical
results are available: extension of both Noether symmetry
theorems for first-order problems \cite{Georgieva2002,Georgieva2005,Georgieva2003};
Euler--Lagrange and transversality optimality conditions for higher-order
variational problems of Herglotz type \cite{MyArt01}; the first Noether theorem
for first-order problems of Herglotz with time delay  \cite{MyArt02};
and the first Noether theorem for higher-order problems of Herglotz type \cite{MyArt04}.
For an optimal control approach to first-order Herglotz type problems see \cite{MyArt03}.

Dynamic systems with time delay are very important in modelling real-life phenomena 
in several fields, such as mathematics, biology, chemistry, economics and mechanics. 
Indeed, several process outcomes are determined not only by variables at present time, 
but also by its behaviour in the past. Motivated by the importance of problems with 
time delay, many authors generalized the classical results of the calculus of variations 
to the delayed case. The first work in this direction seems to have been published
by \`El'sgol'c \cite{MR0170048}. For some recent works on optimal control
see \cite{MR3259239,MR3244467} and references therein. The importance 
of variational problems of Herglotz, as well as the wide applicability 
of problems with time delay, allied to the impossibility of applying 
the classical Noether theorem to these problems, constitute the main 
motivation to the present work.

The main goal of this paper is to generalize the results
of \cite{MyArt01,MyArt02,MyArt03,MyArt04} by
considering higher-order variational problems of Herglotz type with a time delay,
proving the corresponding Euler--Lagrange equations, transversality conditions,
the DuBois--Reymond necessary optimality condition and Noether's first theorem.
In particular, in relation to our previous work with time delay \cite{MyArt02}, 
we improved its results by considering a wider class of admissible functions. 
Moreover, we extend the results of \cite{MyArt02} to the higher-order case.
Precisely, we generalize Herglotz's problem \eqref{PH} by considering
the following variational problem with time delay.
\begin{problem}
\textit{Let $\tau$ be a real number such that $0\leq \tau<b-a$.
Determine the piecewise trajectories $x\in PC^n([a-\tau,b];\mathbb{R}^m)$
and the function $z\in PC^1([a,b];\mathbb{R})$ such that:}
\begin{equation*}
z(b)\longrightarrow \textrm{extr},\\
\end{equation*}
\textit{where the pair} $(x(\cdot),z(\cdot))$
\textit{satisfies the differential equation}
\begin{equation*}
\dot{z}(t)=L\left(t,x(t),\dot{x}(t),\dots, x^{(n)}(t),
x(t-\tau),\dot{x}(t-\tau),\dots, x^{(n)}(t-\tau),z(t)\right),
\end{equation*}
\textit{for} $t \in [a,b]$, \textit{and is subject to initial conditions}
\begin{equation*}
z(a)=\gamma \in \mathbb{R} \text{ and }
x^{(k)}(t)=\mu^{(k)}(t), \quad k=0,\dots,n-1,
\end{equation*}
\textit{where} $\mu \in PC^n ([a-\tau,a];\mathbb{R}^m)$
\textit{ is a given initial function. The Lagrangian
$L$ is assumed to satisfy the following hypotheses:}
\begin{itemize}
\item[i.] $L \in C^1([a,b]\times\mathbb{R}^{2mn+1}; \mathbb{R})$;

\item[ii.] \textit{functions} $t \mapsto
\frac{\partial L}{\partial z}[x;z]_\tau^n(t)$,
$t\mapsto \frac{\partial L}{\partial x^{(k)}}[x;z]_\tau^n(t)$
\textit{and} $t\mapsto \frac{\partial L}{\partial x_\tau^{(k)}}[x;z]_\tau^n(t)$
\textit{are differentiable for any admissible pair}
$(x(\cdot),z(\cdot))$, $k=0,\dots,n$,
\end{itemize}
\textit{where, to simplify expressions, we use the notation
$x_\tau^{(k)}(t)$, $k=0,\dots, n$, to denote the $k$th derivative
of $x$ evaluated at $t-\tau$ (often we use
$x_\tau(t)$ for $x_\tau^{(0)}(t) = x(t-\tau)$ and
$\dot{x}_\tau(t)$ for $x_\tau^{(1)}(t) = \dot{x}(t-\tau)$) and
$$
[x;z]^n_\tau(t):=\left(t,x(t),\dot{x}(t),
\dots, x^{(n)}(t), x_\tau(t),\dot{x}_\tau(t),\dots, x_\tau^{(n)}(t),z(t)\right).
$$
}
\end{problem}

The structure of the paper is as follows. In Section~\ref{sec:prelim}
we recall the necessary background: the well-known Pontryagin's maximum principle,
the DuBois--Reymond necessary optimality condition, and an extension of the
classical Noether's theorem for optimal control problems. In
Section~\ref{sec:MainRes} we formulate and prove our main results:
the higher-order Euler--Lagrange equations and transversality conditions
for generalized variational problems with time delay (Theorem \ref{thm:E-L});
the DuBois--Reymond optimality condition (Theorem~\ref{thm:DB-R});
and Noether's theorem for higher-order variational problems
of Herglotz type with time delay (Theorem~\ref{thm:Noether}).
We end with Section~\ref{sec:conc} of conclusions and possible
future work.


\section{Preliminaries}
\label{sec:prelim}

We begin by recalling the problem of optimal control in Bolza form:
\begin{equation}
\label{problem P}
\tag{$P$}
\begin{gathered}
\mathcal{J}(x(\cdot),u(\cdot))=\int_a^b f(t,x(t),u(t))dt
+\phi(x(b))\longrightarrow\textrm{extr}\\
\text{subject to } \dot{x}(t)=g(t,x(t),u(t)),
\end{gathered}
\end{equation}
with some initial condition on $x$,
where $f \in C^1([a,b]\times \mathbb{R}^{m}\times \Omega;\mathbb{R})$,
$\phi \in C^1(\mathbb{R}^{m};\mathbb{R})$,
$g \in C^1([a,b]\times \mathbb{R}^{m}\times \Omega;\mathbb{R}^m)$,
$x \in PC^1([a,b]; \mathbb{R}^m)$ and $u\in PC([a,b];\Omega)$,
with $\Omega \subseteq \mathbb{R}^r$ an open set.
Usually $x$ and $u$ are called the state
and control variables, respectively, while $\phi$ is known
as the payoff or salvage term. It is clear that the classical
problem of the calculus of variations
is a particular case of problem \eqref{problem P}
with $\phi(x) \equiv 0$, $g(t,x,u)=u$ and $\Omega=\mathbb{R}^m$.
Next we present Pontryagin's maximum principle,
one of the main tools for this paper.

\begin{theorem}[Pontryagin's maximum principle
for problem \eqref{problem P} \cite{Pontryagin}]
\label{PMP}
If a pair $(x(\cdot),u(\cdot))$ with
$x \in PC^1([a,b]; \mathbb{R}^m)$ and $u\in PC([a,b];\Omega)$
is a solution to problem \eqref{problem P} with the initial condition $x(a)=\alpha$,
$\alpha \in \mathbb{R}^m$, then there exists $\psi \in PC^1([a,b];\mathbb{R}^m)$
such that the following conditions hold:
\begin{itemize}
\item the optimality condition
\begin{equation}
\label{prob P opt condt}
\frac{\partial H}{\partial u}(t, x(t),u(t), \psi(t))=0;
\end{equation}
		
\item the adjoint system
\begin{equation}
\label{prob P adj syst}
\begin{cases}
\dot{x}(t)=\frac{\partial H}{\partial \psi}(t, x(t),u(t), \psi(t))\\
\dot{\psi}(t)=-\frac{\partial H}{\partial x}(t, x(t),u(t), \psi(t));
\end{cases}
\end{equation}
		
\item the transversality condition
\begin{equation}
\label{prob P tr condt}
\psi(b)=grad(\phi(x))(b);
\end{equation}
\end{itemize}
where the Hamiltonian $H$ is defined by
\begin{equation}
\label{eq:def:Hamiltonian}
H(t,x,u,\psi)=f(t,x,u)+\psi\cdot g(t,x,u).
\end{equation}
\end{theorem}

\begin{definition}[Pontryagin extremal to \eqref{problem P}]
A triplet $(x(\cdot),u(\cdot), \psi(\cdot))$ with $x \in PC^1([a,b]$;
$\mathbb{R}^m)$, $u\in PC([a,b];\Omega)$ and
$\psi \in PC^1([a,b];\mathbb{R}^m)$ is called a Pontryagin extremal
to problem \eqref{problem P} if it satisfies the necessary
optimality conditions \eqref{prob P opt condt}--\eqref{prob P tr condt}.
\end{definition}

Now we present the following necessary optimality condition
that is used in the proof of our Theorem~\ref{thm:DB-R}.

\begin{theorem}[DuBois--Reymond condition of optimal control \cite{Pontryagin}]
\label{thm DB-R OC}
If $(x(\cdot),u(\cdot), \psi(\cdot))$ is a Pontryagin extremal
to problem \eqref{problem P}, then the Hamiltonian \eqref{eq:def:Hamiltonian}
satisfies the equality
\begin{equation*}
\frac{d H}{dt}(t,x(t),u(t),\psi(t))
=\frac{\partial H}{\partial t}(t,x(t),u(t),\psi(t)),
\quad 	t \in [a,b].
\end{equation*}
\end{theorem}

Many years before the publication of the celebrated result of Pontryagin
\textit{\textit{et al.}} in \cite{Pontryagin}, Emmy Noether proved two
remarkable theorems that relate the invariance of a variational integral
with the corresponding Euler--Lagrange equations. Several extensions
of the two Noether theorems were proved in different contexts. In this paper
we are concerned with the first Noether theorem, also known simply
by Noether's theorem. The Noether theorem \cite{Noether1918} is a fundamental tool
of the calculus of variations \cite{MR2098297},
optimal control \cite{Torres2001,Torres2002,Torres2004}
and modern theoretical physics \cite{MR3072684}.
This theorem guarantees that when an optimal control problem is invariant
under a one parameter family of transformations,
then there exists a corresponding conservation law:
an expression that is conserved along all the Pontryagin
extremals of the problem (see \cite{Torres2001,Torres2002,Torres2004}
and references therein). Here we use Noether's theorem as stated
in \cite{Torres2001}, which is formulated for problems of optimal control
in Lagrange form, that is, for problem \eqref{problem P} with
$\phi \equiv 0$. In order to apply the results of \cite{Torres2001}
to the Bolza problem \eqref{problem P}, we rewrite
it in the following equivalent Lagrange form:
\begin{equation}
\label{eq:prob:Lag}
\begin{gathered}
\mathcal{I}(x(\cdot),y(\cdot),u(\cdot))
=\int_a^b \left[f(t,x(t),u(t))+ y(t)\right] dt \longrightarrow\textrm{extr},\\
\begin{cases}
\dot{x}(t)=g\left(t,x(t),u(t)\right),\\
\dot{y}(t)=0,
\end{cases}\\
x(a)=\alpha, \ y(a)= \frac{\phi(x(b))}{b-a}.
\end{gathered}
\end{equation}
To present Noether's theorem for the optimal control problem
\eqref{problem P}, we need to define the concept of invariance.
In this paper we follow the definition of invariance found
in \cite{Torres2001} applied to the equivalent optimal control
problem \eqref{eq:prob:Lag}. In Definition~\ref{def inv OC}
we use the little-o notation.

\begin{definition}[Invariance of problem \eqref{problem P}]
\label{def inv OC}
Let $h^s$ be a one-parameter family of invertible $C^1$ maps
\begin{equation*}
\begin{gathered}
h^s:[a,b]\times \mathbb{R}^m\times \Omega
\longrightarrow\mathbb{R}\times\mathbb{R}^m\times \mathbb{R}^r,\\
h^s(t,x,u)=\left(\mathcal{T}^s(t,x,u),
\mathcal{X}^s(t,x,u),\mathcal{U}^s(t,x,u)\right),\\
h^0(t,x,u)=(t,x,u) \text{ for all }
(t,x,u)\in [a,b]\times \mathbb{R}^m\times \Omega.
\end{gathered}
\end{equation*}
Problem \eqref{problem P} is said to be invariant under transformations $h^s$
if for all $(x(\cdot),u(\cdot))$ the following two conditions hold:
\begin{multline*}
\left[f \circ h^s(t,x(t),u(t))+\frac{\phi(x(b))}{b-a} + \xi s
+ o(s)\right]\frac{d\mathcal{T}^s}{dt}(t,x(t),u(t))\\
= f(t,x(t),u(t)) + \frac{\phi(x(b))}{b-a}
\end{multline*}
for some constant $\xi$; and
\begin{equation*}
\frac{d\mathcal{X}^s}{dt}\left(t,x(t),u(t)\right)
=g\circ h^s(t,x(t),u(t))\frac{d\mathcal{T}^s}{dt}(t,x(t),u(t)).
\end{equation*}
\end{definition}

As a direct consequence of Noether's theorem proved
in \cite{Torres2001} and Pontryagin's maximum principle
(Theorem~\ref{PMP}), we get the following result that
is central to prove our Theorem~\ref{thm:Noether}.

\begin{theorem}[Noether's theorem for the optimal control problem \eqref{problem P}]
\label{thm opt cont Noether}
If problem \eqref{problem P} is invariant
in the sense of Definition~\ref{def inv OC}, then the quantity
\begin{multline*}
(b-t) \xi + \psi(t) \cdot X(t,x(t),u(t))\\
-\left[H(t,x(t),u(t),\psi(t)) + \frac{\phi(x(b))}{b-a}\right]\cdot T(t,x(t),u(t))
\end{multline*}
is constant in $t$ along every Pontryagin extremal $(x(\cdot),u(\cdot), \psi(\cdot))$
of problem \eqref{problem P}, where $H$ is defined by \eqref{eq:def:Hamiltonian} and
\begin{equation*}
\begin{gathered}
T(t,x(t),u(t))=\frac{\partial \mathcal{T}^s}{\partial s}(t,x(t),u(t))\biggm\vert_{s=0},\\
X(t,x(t),u(t))=\frac{\partial \mathcal{X}^s}{\partial s}(t,x(t),u(t))\biggm\vert_{s=0}.
\end{gathered}
\end{equation*}
\end{theorem}


\section{Main results}
\label{sec:MainRes}

We begin by introducing some definitions for the variational problem
of Herglotz with time delay $\mathbf{(H_\tau^n)}$.

\begin{definition}[Admissible pair to problem $\mathbf{(H_\tau^n)}$]
We say that $(x(\cdot),z(\cdot))$ with $x(\cdot) \in PC^n([a-\tau,b];\mathbb{R}^m)$
and $z(\cdot) \in PC^1([a,b];\mathbb{R})$ is an admissible pair to problem
$\mathbf{(H_\tau^n)}$ if it satisfies the equation
$$
\dot{z}(t)=L[x;z]^n_\tau(t), \quad t \in [a,b],
$$
subject to
$$
z(a)=\gamma, x^{(k)}(t)=\mu^{(k)}(t)
$$
for all $k=0,1,\dots,n-1$, $t\in [a-\tau,a]$
and $\gamma \in \mathbb{R}$.
\end{definition}

\begin{definition}[Extremizer to problem $\mathbf{(H_\tau^n)}$]
We say that an admissible pair $(x^*(\cdot),z^*(\cdot))$ is an extremizer
to problem $\mathbf{(H_\tau^n)}$ if $z(b)-z^*(b)$ has the same signal for all
admissible pairs $(x(\cdot),z(\cdot))$ that satisfy
$\|z-z^* \|_0< \epsilon$ and $\|x-x^* \|_0< \epsilon$
for some positive real $\epsilon$, where
$\|y\|_0=\smash{\displaystyle\max_{a\leq t \leq b}}|y(t)|$.
\end{definition}


\subsection{Reduction to a non-delayed problem}

We generalize the technique of reduction of a delayed first-order
optimal control problem to a non-delayed problem proposed
by Guinn in \cite{Guinn} to our higher-order delayed problem.
In order to reduce the higher-order problem of Herglotz with time delay
to a non-delayed first-order problem, we assume,
without loss of generality, the initial time to be zero ($a=0$)
and the final time to be an integer multiple of $\tau$, that is,
$b=N\tau$ for $N \in \mathbb{N}$ (see Remark~\ref{rem:red:bN}).
We divide the interval $[a,b]$ into $N$ equal parts and fix $t \in [0,\tau]$.
We also introduce the variables $x^{k;i}$ and $z_j$ with $k=0,\dots,n$,
$i=0,\dots,N$, and $j=1,\dots, N+1$. The variable $k$ is related to the order
of the derivative of $x$, $i$ is related to the $i$th subinterval of
$[-\tau,N\tau]$, and $j$ is related to the $j$th subinterval
of $[0,(N+1)\tau]$ as follows:
\begin{equation}
\label{reduction}
\begin{gathered}
x^{k;i}(t)=x^{(k)}(t+(i-1)\tau),
\quad z_j(t)=z(t+(j-1)\tau),\\
\dot{z}_j(t)=L_j(t),\quad
x^{k;N+1}(t)=0,
\quad \dot{z}_{N+1}(t)=L_{N+1}=0
\end{gathered}
\end{equation}
with
$$
L_j(t):=L\left(t+(j-1)\tau,x^{0;j}(t),\dots,x^{n;j}(t),
x^{0;j-1}(t),\dots,x^{n;j-1}(t),z_j(t)\right).
$$
Finally, the higher-order problem of Herglotz with time delay $\mathbf{(H_\tau^n)}$
can be written as an optimal control problem without time delay as follows:
\begin{equation}
\label{Problem OC form}
\begin{gathered}
z_N(\tau)\longrightarrow\textrm{extr},\quad \text{subject to }\\
\begin{cases}
\dot{x}^{k;i}(t)=x^{k+1;i}(t),  \\
x^{k;N+1}(t)=0,\\
\dot{z}_j(t)=L_j(t), \\
\dot{z}_{N+1}(t)=L_{N+1}(t)=0
\end{cases}\\
\text{ for all } t \in[0,\tau] \text{ and with the initial conditions}\\
x^{k;0}(0)=\mu^{(k)}(-\tau),\quad
x^{k;i}(0)=x^{k;i-1}(\tau),\\
z_1(0)=\gamma,\quad \gamma \in \mathbb{R},
\quad z_j(0)=z_{j-1}(\tau)
\end{gathered}	
\end{equation}
for $k=0, \dots, n-1$, $i=0,\dots, N$ and $j=1,\dots, N$.
In this form we look to $x^{k;i}$ and $z_j$ as state variables
and to $u_i:=x^{n;i}$ as the control variables.

\begin{remark}
\label{rem:red:bN}
We considered the case of $b$ being an integer multiple of $\tau$.
If $b$ is not an integer multiple of $\tau$, then there is an integer $N$
such that $(N-1)\tau<b<N\tau$. In that case, the only modification required
in the change of variables given in \eqref{reduction} is to consider the variables
$x^{k;N}$, $k=0, \dots, n$, and $\dot{z}_N$ as defined in \eqref{reduction}
for $t \in [0,b-(N-1)\tau]$ and zero for $t\in ]b-(N-1)\tau, \tau]$. With this
slight change, the function to be extremized remains the same and we can consider,
without loss of generality, $b$ to be an integer multiple of $\tau$.
\end{remark}


\subsection{Higher-order Euler--Lagrange and DuBois--Reymond
optimality conditions with time delay}

We begin by proving a necessary condition for a pair $(x(\cdot),z(\cdot))$
to be an extremizer to problem $\mathbf{(H_\tau^n)}$. Along the proofs
we sometimes suppress arguments for expressions whose arguments
have been clearly stated before.

\begin{theorem}[Higher-order delayed Euler--Lagrange and transversality conditions]
\label{thm:E-L}
If $(x(\cdot),z(\cdot))$ is an extremizer to problem $\mathbf{(H_\tau^n)}$
that satisfies the conditions $x^{(k)}(t)=\mu^{(k)}(t)$,
$k=0,\dots,n-1$ and $t\in [a-\tau, a]$, with $\mu \in PC^n ([a-\tau,a];
\mathbb{R}^m)$, then the two Euler--Lagrange equations
\begin{equation}
\label{EL1}
\sum_{l=0}^{n}(-1)^l\frac{d^l}{dt^l}\left(\psi_z(t)\frac{\partial L}{\partial
x^{(l)}}[x;z]^n_\tau(t)+\psi_z(t+\tau)\frac{\partial L}{\partial
x_\tau^{(l)}}[x;z]^n_\tau(t+\tau)\right)=0,
\end{equation}
for $t \in [a,b-\tau]$, and
\begin{equation}
\label{EL2}
\sum_{l=0}^{n}(-1)^l\frac{d^l}{dt^l}\left(\psi_z(t)
\frac{\partial L}{\partial x^{(l)}}[x;z]^n_\tau(t)\right)=0,
\end{equation}
for $t \in [b-\tau,b]$ and $\psi_z$ defined by
\begin{equation*}
\psi_z(t)=e^{\int_t^b\frac{\partial L}{\partial z}[x;z]^n_\tau(\theta)d\theta},
\quad t \in [a,b],
\end{equation*}
hold. Furthermore, the following transversality conditions hold:
\begin{equation}
\label{TC}
\sum_{l=0}^{n-k}(-1)^l\frac{d^l}{dt^l}\left(\psi_z(t)
\frac{\partial L}{\partial x^{(l+k)}}[x;z]^n_\tau(t)\right)\bigg\vert_{t=b}=0,
\end{equation}
$k=1,\dots,n$.
\end{theorem}

\begin{proof}
In order to prove both Euler--Lagrange equations consider problem
$\mathbf{(H_\tau^n)}$ in the optimal control form \eqref{Problem OC form}.
Applying Pontryagin's maximum principle for problem \eqref{problem P} to problem
$\mathbf{(H_\tau^n)}$ in the form \eqref{Problem OC form}, we conclude that there
are multipliers $\phi_{k;i}$ and $\psi_j$ for $k= 1, \dots, n$, $i=0,\dots, N$
and $j=1,\dots, N+1$, such that, with the Hamiltonian defined by
\begin{equation}
\label{hamilt}
H=\sum_{l=1}^{n}\left(\sum_{i=0}^{N}\phi_{l;i}(t)
\cdot x^{l;i}(t)\right)+\sum_{j=1}^{N+1}\psi_j(t)L_j(t),
\end{equation}
the following conditions hold:
\begin{itemize}
\item the optimality conditions
\begin{equation*}
\frac{\partial H}{\partial u_i}=0,
\end{equation*}

\item the adjoint system
\begin{equation*}
\begin{cases}
\dot{x}^{k-1;i}=\frac{\partial H}{\partial \phi_{k;i}},\\
\dot{z}_j=\frac{\partial H}{\partial \psi_j},\\
\dot{\phi}_{k;i}=-\frac{\partial H}{\partial x^{k-1;i}},\\
\dot{\psi_j}=-\frac{\partial H}{\partial z_j},
\end{cases}
\end{equation*}
\item the transversality conditions
\begin{equation*}
\begin{cases}
\phi_{k;i}(\tau)=0,\\
\psi_j(\tau)=1.
\end{cases}
\end{equation*}
\end{itemize}
\noindent Observe that the forth equation in the adjoint system is equivalent
to the differential equation $\dot{\psi_j}=-\psi_j\frac{\partial L_j}{\partial z_j}$.
Together with the transversality condition, we obtain that
the multipliers $\psi_j$, $j=1, \dots, N+1$, are given by
\begin{equation*}
\psi_j(t)=e^{\int_t^\tau\frac{\partial L_j}{\partial z_j}d\theta}.
\end{equation*}
From the third equation in the adjoint system, we obtain that
\begin{equation}
\label{phi_derivative}
\dot{\phi}_{k;i}=-\phi_{k-1;i}-\psi_i\frac{\partial L_i}{\partial x^{k-1;i}}
-\psi_{i+1}\frac{\partial L_{i+1}}{\partial x^{k-1;i}},
\end{equation}
$k, i=1,\dots,n$, which for the particular case of $k=n$ reduces to
\begin{equation*}
\dot{\phi}_{n;i}=-\phi_{n-1;i}-\psi_i\frac{\partial L_i}{\partial x^{n-1;i}}
-\psi_{i+1}\frac{\partial L_{i+1}}{\partial x^{n-1;i}}.
\end{equation*}
This equality, together with the
differentiation of the optimality condition
\begin{equation*}
\begin{aligned}
\dot{\phi}_{n;i}=&-\frac{d}{dt}\left(\psi_i\frac{\partial L_i}{\partial u_i}\right)
-\frac{d}{dt}\left(\psi_{i+1}\frac{\partial L_{i+1}}{\partial u_i}\right)\\
=&-\frac{d}{dt}\left(\psi_i\frac{\partial L_i}{\partial x^{n;i}}+\psi_{i+1}
\frac{\partial L_{i+1}}{\partial x^{n;i}}\right),
\end{aligned}
\end{equation*}
leads to
\begin{equation*}
\phi_{n-1;i}=-\psi_i\frac{\partial L_i}{\partial x^{n-1;i}}
-\psi_{i+1}\frac{\partial L_{i+1}}{\partial x^{n-1;i}}
+\frac{d}{dt}\left(\psi_i\frac{\partial L_i}{\partial x^{n;i}}
+\psi_{i+1}\frac{\partial L_{i+1}}{\partial x^{n;i}}\right).
\end{equation*}
By differentiation of the previous expression and comparison with
\eqref{phi_derivative} for $k=n-1$, we find the expression
for $\phi_{n-2;i}$:
\begin{multline*}
\phi_{n-2;i}=-\psi_i\frac{\partial L_i}{\partial x^{n-2;i}}
-\psi_{i+1}\frac{\partial L_{i+1}}{\partial x^{n-2;i}}\\
+\frac{d}{dt}\left(\psi_i\frac{\partial L_i}{\partial x^{n-1;i}}
+\psi_{i+1}\frac{\partial L_{i+1}}{\partial x^{n-1;i}}\right)
-\frac{d^2}{dt^2}\left(\psi_i\frac{\partial L_i}{\partial x^{n;i}}
+\psi_{i+1}\frac{\partial L_{i+1}}{\partial x^{n;i}}\right).
\end{multline*}
Using recursively the technique of derivation of $\phi_{k;i}$ and comparison
with \eqref{phi_derivative}, we find the expression
for $\phi_{k;i}$ ($k=1, \dots, n$):
\begin{equation}
\label{phi_expression}
\phi_{k;i}=\sum_{l=0}^{n-k}(-1)^{{l+1}}\frac{d^l}{dt^l}\left(\psi_i
\frac{\partial L_i}{\partial x^{l+k;i}}+\psi_{i+1}
\frac{\partial L_{i+1}}{\partial x^{l+k;i}}\right),\quad i=1, \dots, N.
\end{equation}
Considering $\phi_{1;i}$ given by the previous equation and comparing it with
\begin{equation*}
\phi_{1;i}=-\dot{\phi}_{2;i}-\psi_i\frac{\partial L_i}{\partial x^{1;i}}
-\psi_{i+1}\frac{\partial L_{i+1}}{\partial x^{1;i}},
\end{equation*}
given by \eqref{phi_derivative} for $k=2$, we obtain that
\begin{equation}
\label{EL1_proof}
\sum_{l=0}^{n}(-1)^l\frac{d^l}{dt^l}\left(\psi_i\frac{\partial
L_i}{\partial x^{l;i}}+\psi_{i+1}\frac{\partial L_{i+1}}{\partial x^{l;i}}\right)
=0,\quad i=1, \dots, N.
\end{equation}
Since $L_{N+1}=0$, the previous equation for $i=N$ reduces to
\begin{equation}
\label{EL2_proof}
\sum_{l=0}^{n}(-1)^l\frac{d^l}{dt^l}\left(\psi_N
\frac{\partial L_N}{\partial x^{l;N}}\right)=0.
\end{equation}
The final step is to rewrite the results obtained inverting the changes
of variables \eqref{reduction}. For this purpose, define
$\psi_z(t)$, $t \in [0, b+\tau]$, by
\begin{equation*}
\psi_z(t)=\psi_i(t-(i-1)\tau), \quad (i-1)\tau\leq t \leq i\tau,
\quad i = 1, \ldots, N+1,
\end{equation*}
and $\phi_k(t)$, $k=1,\dots,n$, $t \in [-\tau, b]$, by
\begin{equation*}
\label{multipliers_phi}
\phi_k(t)= \phi_{k;i}(t-(i-1)\tau),
\quad (i-1)\tau\leq t \leq i \tau,
\quad i = 1, \ldots, N.
\end{equation*}
This allows to write
\begin{equation}
\label{psi}
\psi_z(t)=e^{\int_t^b\frac{\partial L}{\partial z}[x;z]^n_\tau(\theta)d\theta},
\quad t \in [a,b],
\end{equation}
and
\begin{equation}
\label{phi}
\begin{gathered}
\phi_{k}(t)=\sum_{l=0}^{n-k}(-1)^{l+1}\frac{d^l}{dt^l}\left(\psi_{z}(t+\tau)
\frac{\partial L}{\partial x_\tau^{(l+k)}}[x;z]^n_\tau(t+\tau)\right),
\quad t \in [a-\tau,a],\\
\phi_{k}(t)=\sum_{l=0}^{n-k}(-1)^{l+1}\frac{d^l}{dt^l}\Biggl(\psi_z(t)
\frac{\partial L}{\partial x^{(l+k)}}[x;z]^n_\tau(t)\\
\qquad\quad +\psi_{z}(t+\tau)
\frac{\partial L}{\partial x_\tau^{(l+k)}}[x;z]^n_\tau(t+\tau)\Biggr),
\quad t\in [a,b],
\end{gathered}
\end{equation}
$k=1, \dots, n$. Note that if $t\in[b-\tau, b]$, then $L[x;z]_\tau^n(t+\tau)$
is, by definition, null. Finally, equations \eqref{EL1_proof}--\eqref{EL2_proof} 
lead to the Euler--Lagrange equations for the higher-order problem of Herglotz
with time delay $\mathbf{(H_\tau^n)}$:
\begin{equation*}
\sum_{l=0}^{n}(-1)^l\frac{d^l}{dt^l}\left(\psi_z(t)\frac{\partial L}{\partial
x^{(l)}}[x;z]^n_\tau(t)+\psi_z(t+\tau)\frac{\partial L}{\partial
x_\tau^{(l)}}[x,z]^n_\tau(t+\tau)\right)=0
\end{equation*}
for $t \in [a,b-\tau]$ and
\begin{equation*}
\sum_{l=0}^{n}(-1)^l\frac{d^l}{dt^l}\left(\psi_z(t)
\frac{\partial L}{\partial x^{(l)}}[x;z]^n_\tau(t)\right)=0
\end{equation*}
for $t \in [b-\tau,b]$. From \eqref{phi_expression} and the transversality
conditions for $\phi_{k;i}$, we obtain the transversality
conditions $\phi_k(b)=0$, that is,
\begin{equation*}
\sum_{l=0}^{n-k}(-1)^l\frac{d^l}{dt^l}\left(\psi_z(t)\frac{\partial
L}{\partial x^{(l+k)}}[x;z]^n_\tau(t)\right)\bigg\vert_{t=b}=0,
\end{equation*}
$k=1, \dots, n$.
\end{proof}

\begin{definition}[Extremal to problem $\mathbf{(H_\tau^n)}$]
We say that an admissible pair $(x(\cdot),z(\cdot))$ is an extremal to problem
$\mathbf{(H_\tau^n)}$ if it satisfies the Euler--Lagrange equations
\eqref{EL1}--\eqref{EL2} and the transversality conditions \eqref{TC}.
\end{definition}

Theorem~\ref{thm:E-L} gives a generalization of the Euler--Lagrange equation
and transversality conditions for the higher-order problem of Herglotz
presented by the authors in \cite{MyArt01}. It is also a generalization
of the results in \cite{MyArt03,MyArt04}.

\begin{corollary}[cf. \cite{MyArt01,MyArt04}]
\label{cor:E-L tau=0}
If $(x(\cdot),z(\cdot))$ is an extremizer
to the higher-order problem of Herglotz
\begin{equation}
\label{eq:prb:cor:E-L tau=0}
\begin{gathered}
z(b)\longrightarrow \textrm{extr},\\
\dot{z}(t)=L\left(t,x(t),\dot{x}(t),\dots, x^{(n)}(t),z(t)\right),
\quad t \in [a,b],\\
z(a)=\gamma \in \mathbb{R}, \quad x^{(k)}(a)=\alpha_k, \quad \alpha_k
\in \mathbb{R}^m, \quad k=0,\dots,n-1,
\end{gathered}
\end{equation}
then the Euler--Lagrange equation
\begin{equation*}
\sum_{l=0}^{n}(-1)^l\frac{d^l}{dt^l}\left(\psi_z(t)
\frac{\partial L}{\partial x^{(l)}}[x;z]^n_0(t)\right)=0
\end{equation*}
holds for  $t \in [a,b]$, where $\psi_z$ is defined in \eqref{psi}.
Furthermore, the following transversality conditions hold:
\begin{equation*}
\sum_{l=0}^{n-k}(-1)^l\frac{d^l}{dt^l}\left(\psi_z(t)\frac{\partial
L}{\partial x^{(l+k)}}[x;z]^n_0(t)\right)\bigg\vert_{t=b}=0,
\end{equation*}
$k=1,\dots,n$.
\end{corollary}

\begin{proof}
Consider Theorem~\ref{thm:E-L} with no delay,
that is, with $\tau=0$.
\end{proof}

Theorem~\ref{thm:E-L} is also a generalization of the Euler--Lagrange equations
for the first-order problem of Herglotz with time delay obtained in \cite{MyArt02}.

\begin{corollary}[cf. \cite{MyArt02}]
\label{cor:E-L n=1}
If $(x(\cdot),z(\cdot))$ is an extremizer to the first-order problem
of Herglotz with time delay
\begin{equation}
\label{eq:cor:E-L n=1}
\begin{gathered}
z(b)\longrightarrow \textrm{extr},\\
\dot{z}(t)=L\left(t,x(t),\dot{x}(t),x(t-\tau),\dot{x}(t-\tau),z(t)\right),
\quad t \in [a,b], \\
z(a)=\gamma \in \mathbb{R}, \quad x(t)=\mu(t), \quad t \in [a-\tau,a],
\end{gathered}
\end{equation}
for a given piecewise initial function $\mu$,
then the Euler--Lagrange equations
\begin{multline*}
\psi_z(t)\frac{\partial L}{\partial x}[x;z]^1_\tau(t)+\psi_z(t+\tau)
\frac{\partial L}{\partial x_\tau}[x,z]^1_\tau(t+\tau)\\
-\frac{d}{dt}\left(\psi_z(t)\frac{\partial L}{\partial \dot{x}}[x;z]^1_\tau(t)
+\psi_z(t+\tau)\frac{\partial L}{\partial \dot{x}_\tau}[x,z]^1_\tau(t+\tau)\right)=0,
\end{multline*}
for $t \in [a,b-\tau]$, and
\begin{equation*}
\psi_z(t)\frac{\partial L}{\partial x}[x;z]^1_\tau(t)
-\frac{d}{dt}\left(\psi_z(t)
\frac{\partial L}{\partial \dot{x}}[x;z]^1_\tau(t)\right)=0,
\end{equation*}
for $t \in [b-\tau,b]$, hold.
\end{corollary}

\begin{proof}
Consider Theorem~\ref{thm:E-L} with $n=1$.
\end{proof}

\begin{theorem}[Higher-order delayed DuBois--Reymond condition]
\label{thm:DB-R}
If the pair $(x(\cdot), z(\cdot))$ is an extremal
to problem $\mathbf{(H_\tau^n)}$, then
\begin{equation}
\label{eq:HODR:cor}
\frac{d}{dt}\left(\sum_{k=1}^n\phi_k(t) \cdot x^{(k)}(t)
+\psi_z(t) L[x;z]^n_\tau(t) \right)
=\psi_z(t)\frac{\partial L}{\partial t}[x;z]^n_\tau(t),
\end{equation}
where $\psi_z$ and $\phi_k$ are defined by \eqref{psi}
and \eqref{phi}, respectively.
\end{theorem}

\begin{proof}
Consider problem $\mathbf{(H_\tau^n)}$ in the formulation
given by \eqref{Problem OC form}. Theorem~\ref{thm DB-R OC} asserts that
$\frac{dH}{dt}=\frac{\partial H}{\partial t}$ for $H$ given by \eqref{hamilt}.
We obtain \eqref{eq:HODR:cor} by writing $H$ in the variables $\phi_k$ and $\psi_z$.
\end{proof}

Theorem~\ref{thm:DB-R} is also a generalization of the DuBois--Reymond condition
presented in \cite{MyArt02} for the first-order problem of Herglotz with time delay.
In that paper, for technical reasons, we added an additional hypothesis
that we are able to avoid here.

\begin{corollary}[cf. \cite{MyArt02}]
\label{cor:DR-R n=1}
If $(x(\cdot),z(\cdot))$ is an extremizer to the first-order problem
of Herglotz with time delay \eqref{eq:cor:E-L n=1}, then
\begin{multline*}
\psi_z(t)\frac{\partial L}{\partial t}[x;z]^1_\tau(t)
=\frac{d}{dt}\Biggl(\psi_z(t) L[x;z]^1_\tau(t)\\
-\left(\psi_z(t)\frac{\partial L}{\partial
\dot{x}}[x;z]^1_\tau(t)+\psi_z(t+\tau)\frac{\partial L}{\partial
\dot{x}_\tau}[x;z]^1_\tau(t+\tau)\right) \dot{x}(t)\Biggr),
\end{multline*}
where $\psi_z$ is defined by \eqref{psi}.
\end{corollary}

\begin{proof}
Consider Theorem~\ref{thm:DB-R} with $n=1$.
\end{proof}


\subsection{Higher-order Noether's symmetry theorem with time delay}

Before presenting a Noether theorem to problem $\mathbf{(H_\tau^n)}$,
we introduce the notion of invariance.

\begin{definition}[Invariance of problem $\mathbf{(H_\tau^n)}$]
\label{def inv_PH_n}
Let $h^s$ be a one-parameter family of invertible $C^1$ maps
$h^s:[a-\tau,b]\times \mathbb{R}^m \times \mathbb{R}
\longrightarrow \mathbb{R}\times \mathbb{R}^m \times \mathbb{R}$,
\begin{equation*}
\begin{gathered}
h^s(t,x(t),z(t))=(\mathcal{T}^s[x;z]^n_\tau(t),
\mathcal{X}^s[x;z]^n_\tau(t),\mathcal{Z}^s[x;z]^n_\tau(t)),\\
h^0(t,x,z)=(t,x,z), \quad \forall (t,x,z) \in [a-\tau,b]
\times \mathbb{R}^m  \times \mathbb{R}.
\end{gathered}
\end{equation*}
Problem $\mathbf{(H_\tau^n)}$ is said to be invariant
under the transformations $h^s$, if for all admissible pairs
$(x(\cdot),z(\cdot))$ the following two conditions hold:
\begin{equation}
\label{eq inv PH_n_1}
\left(\frac{z(b)}{b-a}+\xi s + o(s)\right)
\frac{d\mathcal{T}^s}{dt}[x;z]^n_\tau(t)
=\frac{z(b)}{b-a}
\end{equation}
for some constant $\xi$ and		
\begin{equation}
\begin{split}
\frac{d \mathcal{Z}^s}{dt}[x;z]^n_\tau(t)
&= \frac{d\mathcal{T}^s}{dt}[x;z]^n_\tau(t) \,
L\Biggl(\mathcal{T}^s[x;z]^n_\tau(t),\mathcal{X}^s[x;z]^n_\tau(t),\\
&\ \frac{d\mathcal{X}^s}{d\mathcal{T}^s}[x;z]^n_\tau(t),\ldots,
\frac{d^n\mathcal{X}^s}{d(\mathcal{T}^s)^n}[x;z]^n_\tau(t),
\mathcal{X}^s[x,z]^n_\tau(t-\tau),\\
&\ \frac{d\mathcal{X}^s}{d\mathcal{T}^s}[x,z]^n_\tau(t-\tau),\ldots,
\frac{d^n\mathcal{X}^s}{d(\mathcal{T}^s)^n}[x,z]^n_\tau(t-\tau),
\mathcal{Z}^s[x;z]^n_\tau(t)\Biggr),
\end{split}
\end{equation}
where
\begin{equation*}
\begin{split}
\frac{d\mathcal{X}^s}{d\mathcal{T}^s}[x;z]^n_\tau(t)
&=\frac{\frac{d\mathcal{X}^s}{dt}[x;z]^n_\tau(t)}{
\frac{d\mathcal{T}^s}{dt}[x;z]^n_\tau(t)},\\
\frac{d^k\mathcal{X}^s}{d(\mathcal{T}^s)^k}[x;z]^n_\tau(t)
&=\frac{\frac{d}{dt}\left(\frac{d^{k-1}\mathcal{X}^s}{d(\mathcal{T}^s)^{k-1}}
[x;z]^n_\tau(t)\right)}{\frac{d\mathcal{T}^s}{dt}[x;z]^n_\tau(t)},
\end{split}
\end{equation*}
$k=2,\ldots,n$.
\end{definition}

Now we generalize the higher-order Noether's theorem of \cite{MyArt04}
to the more general case of variational problems of Herglotz type with time delay.

\begin{theorem}[Higher-order delayed Noether's theorem]
\label{thm:Noether}
If problem $\mathbf{(H_\tau^n)}$ is invariant in the sense of
Definition~\ref{def inv_PH_n}, then the quantity
\begin{multline*}
\sum_{k=1}^{n}\phi_k(t)\cdot X_{k-1}[x;z]^n_\tau(t)+\psi_z(t) Z[x;z]^n_\tau(t)\\
-\left[	\sum_{k=1}^{n}\phi_k(t)\cdot x^{(k)}(t)
+\psi_z(t)L[x;z]^n_\tau(t)\right]T[x;z]^n_\tau(t)
\end{multline*}
is constant in $t$ along all extremals of problem $\mathbf{(H_\tau^n)}$,
where the generators of the one-parameter family of maps are given by
\begin{equation*}
\begin{gathered}
T=\frac{\partial \mathcal{T}^s}{\partial s}\biggm\vert_{s=0},
\quad X_0=\frac{\partial \mathcal{X}^s}{\partial s}\biggm\vert_{s=0},
\quad Z=\frac{\partial \mathcal{Z}^s}{\partial s}\biggm\vert_{s=0},\\
X_k=\frac{d }{dt}X_{k-1}-x^{(k)}\frac{d}{dt}\left(
\frac{\partial\mathcal{T}^s}{\partial s}
\bigg\vert_{s=0}\right),
\quad k=1,\ldots, n-1,
\end{gathered}
\end{equation*}
and $\psi_z$, $\phi_k$ are defined by \eqref{psi}--\eqref{phi}.
\end{theorem}

\begin{proof}
We start by considering problem $\mathbf{(H_\tau^n)}$ in its non-delayed
optimal control form \eqref{Problem OC form}. The first step is to prove that
if problem $\mathbf{(H_\tau^n)}$ is invariant in the sense of
Definition~\ref{def inv_PH_n}, then \eqref{Problem OC form} is invariant in the sense
of Definition~\ref{def inv OC}. In order to do that, observe that
\eqref{eq inv PH_n_1} is equivalent to
\begin{equation*}
\left(\frac{z_N(\tau)}{N\tau}+\xi s
+ o(s)\right)\frac{d\mathcal{T}^s}{dt}[x;z]^n_\tau(t)
=\frac{z_N(\tau)}{N\tau}
\end{equation*}
and defining $\xi_\tau:=\xi N$ we have
\begin{equation}
\label{eq inv1 0-tau}
\left(\frac{z_N(\tau)}{\tau}+\xi_\tau s + o(s)\right)
\frac{d\mathcal{T}^s}{dt}[x;z]^n_\tau(t)
=\frac{z_N(\tau)}{\tau}, \quad \text{for some }\xi_\tau.
\end{equation}
Observe also that the control system of \eqref{Problem OC form} defines
$\mathcal{X}_k^s:=\frac{d\mathcal{X}_{k-1}^s}{d\mathcal{T}^s}$, that is,
\begin{equation*}
\frac{d\mathcal{X}_{k-1}^s}{dt}[x;z]^n_\tau(t)
=\mathcal{X}_k^s[x;z]^n_\tau(t)\frac{d\mathcal{T}^s}{dt}[x;z]^n_\tau(t),
\quad k=1, \ldots, n.
\end{equation*}
Let
\begin{equation*}
\begin{split}
\mathcal{X}_{k;i}[x;z]^n_\tau(t)&:=\mathcal{X}_k^s[x;z]^n_\tau(t+(i-1)\tau),\\
\mathcal{T}_{i}[x;z]^n_\tau(t)&:=\mathcal{T}^s[x;z]^n_\tau(t+(i-1)\tau),\\
\mathcal{Z}_j[x;z]^n_\tau(t)&:=\mathcal{Z}^s[x;z]^n_\tau(t+(j-1)\tau).
\end{split}
\end{equation*}
One has
\begin{equation}
\label{eq inv2 0-tau}
\frac{d\mathcal{X}_{k;i}}{dt}[x;z]^n_\tau(t)
=\mathcal{X}_{k+1;i}[x;z]^n_\tau(t)\frac{d\mathcal{T}_{i}}{dt}[x;z]^n_\tau(t)
\end{equation}
and
\begin{multline}
\label{eq inv3 0-tau}
\frac{d \mathcal{Z}_j}{dt}[x;z]^n_\tau(t)
= L_j\left[\mathcal{X}^s[x;z]^n_\tau(t);\mathcal{Z}^s[x;z]^n_\tau(t)
\right]_\tau^n(\mathcal{T}_j^s[x;z]^n_\tau(t))
\frac{d\mathcal{T}_j}{dt}[x;z]^n_\tau(t),
\end{multline}
$k=0,\dots, n-1$, $i=0,\dots N$, $j=1,\dots,N$.
Equalities \eqref{eq inv1 0-tau}--\eqref{eq inv3 0-tau}
prove that problem \eqref{Problem OC form} is invariant
in the sense of Definition \ref{def inv OC}.
This put us in conditions to advance to the second step: to apply 	
Theorem~\ref{thm opt cont Noether} to the non-delayed optimal
control problem \eqref{Problem OC form}. This theorem guarantees
that the quantity
\begin{multline*}
(\tau-t)\xi_\tau+\sum_{k=1}^{n}\sum_{i=0}^{N}\phi_{k;i}(t)
\cdot X_{k-1;i}[x;z]^n_\tau(t)+\sum_{j=1}^{N}\psi_j(t)Z_j[x;z]^n_\tau(t)\\
-\left[	\sum_{k=1}^{n}\sum_{i=0}^{N}\phi_{k;i}(t)\cdot x^{k;i}(t)
+\sum_{j=1}^{N} \psi_j(t) L_j[x;z]^n_\tau(t)	
+\frac{z_N(\tau)}{\tau}\right]T[x;z]^n_\tau(t)
\end{multline*}
is constant in $t$ along the extremals of \eqref{Problem OC form}, where
$X_{k;i}=\frac{\partial}{\partial s} \frac{d^{k}
\mathcal{X}_{k;i}^s}{d(\mathcal{T}^s)^{k}}\Big\vert_{s=0}$
and $Z_{i}=\frac{\partial}{\partial s}
\frac{d \mathcal{Z}_{i}^s}{d(\mathcal{T}^s)}\Big\vert_{s=0}$.
Rewriting in the original variables, we obtain
\begin{multline*}
(\tau-t)\xi_\tau+\sum_{k=1}^{n}\phi_{k}(t)\cdot X_{k-1}[x;z]^n_\tau(t)
+\psi_z(t) Z[x;z]^n_\tau(t)\\
-\left[	\sum_{k=1}^{n}\phi_{k}(t)\cdot x^{(k)}(t)+\psi_z(t)
L[x;z]^n_\tau(t)+\frac{z_N(\tau)}{\tau}\right]T[x;z]^n_\tau(t)
\end{multline*}
constant in $t$ along the extremals of \eqref{Problem OC form}.
The third step is to prove that
\begin{equation}
\label{eq:constant}
(\tau-t)\xi_\tau-\frac{z_N(\tau)}{\tau}T[x;z]^n_\tau(t)
\end{equation}
is constant in $t$. From the invariance condition
\eqref{eq inv1 0-tau}, we know that
\begin{equation*}
\left(\frac{z_N(\tau)}{\tau}+\xi_\tau s
+ o(s)\right)\frac{d\mathcal{T}^s}{dt}[x;z]^n_\tau(t)
=\frac{z_N(\tau)}{\tau}.
\end{equation*}
Integrating from $0$ to $t$ we conclude that
\begin{multline*}
\left(\frac{z_N(\tau)}{\tau}+\xi_\tau s
+ o(s)\right)\mathcal{T}^s[x;z]_\tau^n(t)\\
=\frac{z_N(\tau)}{\tau}t+\left(\frac{z_N(\tau)}{\tau}
+\xi_\tau s + o(s)\right)\mathcal{T}^s[x;z]_\tau^n(0).
\end{multline*}
Differentiating this equality with respect to $s$,
and then putting $s=0$, we get
\begin{equation}
\label{eq:relation}
\xi_\tau t+\frac{z_N(\tau)}{\tau}T[x;z]_\tau^n(t)
=\frac{z_N(\tau)}{\tau}T[x;z]_\tau^n(0).
\end{equation}
We conclude from \eqref{eq:relation} that expression
\eqref{eq:constant} is the constant
$$
\tau\xi_\tau - \frac{z_N(\tau)}{\tau}T[x;z]_\tau^n(0).
$$
Hence,
\begin{multline*}
\sum_{k=1}^{n}\phi_{k}(t)\cdot X_{k-1}[x;z]^n_\tau(t)+\psi_z(t) Z[x;z]^n_\tau(t)\\
-\left[	\sum_{k=1}^{n}\phi_{k}(t)\cdot x^{(k)}(t)+\psi_z(t)
L[x;z]^n_\tau(t)\right]T[x;z]^n_\tau(t)
\end{multline*}
is constant in $t$ along the extremals of problem \eqref{Problem OC form}.
Finally, observe that $X_0=\frac{\partial \mathcal{X}^s}{\partial s}\big\vert_{s=0}$
and
\begin{align*}
X_k&=\frac{\partial}{\partial s}
\frac{d^{k}\mathcal{X}^s}{d(\mathcal{T}^s)^{k}}\bigg\vert_{s=0}
=\frac{\partial}{\partial s} \left( \frac{\frac{d}{dt}\left(
\frac{d^{k-1}\mathcal{X}^s}{d(\mathcal{T}^s)^{k-1}}\right)}{
\frac{d\mathcal{T}^s}{dt}}\right)\Bigg\vert_{s=0}\\
&= \frac{d}{dt}\left(\frac{\partial}{\partial s}
\frac{d^{k-1}\mathcal{X}^s}{d(\mathcal{T}^s)^{k-1}}\bigg\vert_{s=0}\right)-x^{(k)}
\frac{d}{dt}\left(\frac{\partial\mathcal{T}^s}{\partial s} \bigg\vert_{s=0}\right)\\
&=\frac{d }{dt}X_{k-1}-x^{(k)}\frac{d}{dt}\left(
\frac{\partial\mathcal{T}^s}{\partial s} \bigg\vert_{s=0}\right),
\end{align*}
$k=1, \dots, n-1$. This concludes the proof.
\end{proof}

\begin{corollary}[cf. \cite{MyArt04}]
\label{cor:NT tau=0}
If the higher-order problem of Herglotz \eqref{eq:prb:cor:E-L tau=0}
is invariant in the sense of Definition~\ref{def inv_PH_n} (in $[a,b]$),
then the quantity
\begin{multline*}
\sum_{k=1}^{n}\tilde{\phi}_k(t)\cdot X_{k-1}[x;z]^n_0(t)+\psi_z(t) Z[x;z]^n_0(t)\\
-\left[	\sum_{k=1}^{n}\tilde{\phi}_k(t)\cdot x^{(k)}(t)
+\psi_z(t) L[x;z]^n_0(t)\right]T[x;z]^n_0(t)
\end{multline*}
is constant in $t$ along any extremal of the problem, where
\begin{equation*}
\tilde{\phi}_{k}(t)=\sum_{l=0}^{n-k}(-1)^{l+1}
\frac{d^l}{dt^l}\left(\psi_z(t)
\frac{\partial L}{\partial x^{(l+k)}}[x;z]^n_0(t)\right),
\end{equation*}
$k=1, \dots, n$, and $\psi_z$ is given by  \eqref{psi}.
\end{corollary}

\begin{proof}
Consider Theorem~\ref{thm:Noether} with $\tau=0$.
\end{proof}

Theorem~\ref{thm:Noether} is a generalization of Noether's theorem \cite{MyArt02}
for the first-order problem of Herglotz with time delay.
Besides the improvement of dealing with piecewise
functions instead of continuous, the theorem presents
a similar conserved quantity but without the imposition of two additional
hypotheses required in \cite{MyArt02}. Moreover, the current definition
of invariance is more general than the one considered in \cite{MyArt02}.

\begin{corollary}[cf. \cite{MyArt02}]
\label{cor:NT n=1}
If the first-order problem of Herglotz with time delay
\eqref{eq:cor:E-L n=1} is invariant in the sense
of Definition~\ref{def inv_PH_n}, then the quantity
\begin{multline*}
\left(\psi_z(t)\frac{\partial L}{\partial \dot{x}}[x;z]^1_\tau(t)
+\psi_z(t+\tau)\frac{\partial L}{\partial \dot{x}_\tau}[x;z]^1_\tau(t+\tau)\right)
X_0[x;z]^1_\tau(t)\\
+\psi_z(t) Z[x;z]^1_\tau(t)+\left[
-\left(\psi_z(t)\frac{\partial L}{\partial
\dot{x}}[x;z]^1_\tau(t)\right.\right.\\
\left.\left.+\psi_z(t+\tau)\frac{\partial L}{\partial \dot{x}_\tau}[x;z]^1_\tau(t
+\tau)\right) \dot{x}(t)+\psi_z(t) L[x;z]^1_\tau(t)	\right]T[x;z]^1_\tau(t)
\end{multline*}
is constant in $t \in [a,b]$ along any extremal of the problem.
\end{corollary}

\begin{proof}
Consider Theorem~\ref{thm:Noether} with $n=1$.
\end{proof}

\begin{remark}
If $t\in[b-\tau, b]$, then $L[x;z]_\tau^n(t+\tau)$
is, by definition, null (see \eqref{reduction})
and the constant of Corollary~\ref{cor:NT n=1} reduces to
\begin{multline*}
\left(\psi_z(t)\frac{\partial L}{\partial
\dot{x}}[x;z]^1_\tau(t)\right) X_0[x;z]^1_\tau(t)
+\psi_z(t) Z[x;z]^1_\tau(t)\\
+\left[	-\left(\psi_z(t)\frac{\partial L}{\partial \dot{x}}[x;z]^1_\tau(t)\right)
\dot{x}(t)+\psi_z(t) L[x;z]^1_\tau(t)	\right]T[x;z]^1_\tau(t)
\end{multline*}
for $t \in[b-\tau,b]$, which is the second constant quantity of \cite{MyArt02}.
\end{remark}


\section{Conclusion}
\label{sec:conc}

Optimal control is a convenient tool to deal
with delayed and non-delayed Herglotz type
variational problems. In this work we have shown
how some of the central results from the classical
calculus of variations can be proved for higher-order
Herglotz variational problems with time delay
from analogous and well-known optimal control results.
The techniques here developed can now be used to obtain
other results. For example, our optimal control approach
can be employed together with \cite{MR1980565}
to derive an extension of the second Noether theorem
to the delayed or non-delayed Herglotz framework. This
is under investigation and will be addressed elsewhere.


\section*{Acknowledgements}

This research is part of first author's Ph.D. project,
which is carried out at University of Aveiro.
It was partially supported by Portuguese funds through
the Center for Research and Development in Mathematics
and Applications (CIDMA) and the Portuguese Foundation
for Science and Technology (FCT),
within project UID/MAT/04106/2013.
The authors are grateful to an anonymous referee 
for several comments and suggestions, which helped
to improve the quality of the paper.




\begin{thebibliography}{99}

\bibitem{MR3259239}
M. Benharrat\ and\ D. F. M. Torres, 
{\it Optimal control with time delays via the penalty method}, 
Math. Probl. Eng. {\bf 2014}, Art. ID 250419, 9~pp. 
{\tt arXiv:1407.5168}

\bibitem{MR3244467}
A. Debbouche\ and\ D. F. M. Torres, 
{\it Approximate controllability of fractional delay dynamic inclusions 
with nonlocal control conditions}, 
Appl. Math. Comput. {\bf 243} (2014), 161--175. 
{\tt arXiv:1405.6591}

\bibitem{MR0170048}
L. \`E. \`El'sgol'c, 
{\it Qualitative methods in mathematical analysis}, 
Translations of Mathematical Monographs, Vol.~12, 
Amer. Math. Soc., Providence, RI, 1964. 

\bibitem{MR3072684}
G. S. F. Frederico\ and\ D. F. M. Torres,
{\it Fractional isoperimetric Noether's theorem in the Riemann-Liouville sense},
Rep. Math. Phys. {\bf 71} (2013), no.~3, 291--304.
{\tt arXiv:1205.4853}

\bibitem{Georgieva2002}
B. Georgieva\ and\ R. Guenther,
{\it First Noether-type theorem for the generalized variational principle of Herglotz},
Topol. Methods Nonlinear Anal. {\bf 20} (2002), no.~2, 261--273.

\bibitem{Georgieva2005}
B. Georgieva\ and\ R. Guenther,
{\it Second Noether-type theorem for the generalized variational principle of Herglotz},
Topol. Methods Nonlinear Anal. {\bf 26} (2005), no.~2, 307--314.

\bibitem{Georgieva2003}
B. Georgieva, R. Guenther\ and\ T. Bodurov,
{\it Generalized variational principle of Herglotz for several independent variables},
J. Math. Phys. {\bf 44} (2003), no.~9, 3911--3927.

\bibitem{Guenther1996SIAM}
R. B. Guenther, J. A. Gottsch\ and\ D. B. Kramer,
{\it The Herglotz algorithm for constructing canonical transformations},
SIAM Rev. {\bf 38} (1996), no.~2, 287--293.

\bibitem{Guenther1996}
R. B. Guenther, C. M. Guenther\ and\ J. A. Gottsch,
{\it The Herglotz Lectures on Contact Transformations and Hamiltonian Systems},
Lecture Notes in Nonlinear Analysis, Vol.~1,
Juliusz Schauder Center for Nonlinear Studies,
Nicholas Copernicus University, Tor\'{u}n, 1996.
	
\bibitem{Guinn}
T. Guinn,
{\it Reduction of delayed optimal control problems to nondelayed problems},
J. Optimization Theory Appl. {\bf 18} (1976), no.~3, 371--377.

\bibitem{Herglotz1930}
G. Herglotz,
{\it Ber\"uhrungstransformationen},
Lectures at the University of G\"ottingen, G\"ottingen, 1930.

\bibitem{Noether1918}
E. Noether,
{\it Invariante Variationsprobleme},
Nachr. v. d. Ges. d. Wiss. zu Göttingen, 1918, 235--257.

\bibitem{Pontryagin}
L. S. Pontryagin, V. G. Boltyanskii, R. V. Gamkrelidze\ and\ E. F. Mishchenko,
{\it The mathematical theory of optimal processes},
Interscience Publishers, John Wiley and Sons Inc, New York, London, 1962.

\bibitem{MyArt01}
S. P. S. Santos, N. Martins\ and\ D. F. M. Torres,
{\it Higher-order variational problems of Herglotz type},
Vietnam J. Math. {\bf 42} (2014), no.~4, 409--419.
{\tt arXiv:1309.6518}

\bibitem{MyArt02}
S. P. S. Santos, N. Martins\ and\ D. F. M. Torres,
{\it Variational problems of Herglotz type with time delay:
DuBois-Reymond condition and Noether's first theorem},
Discrete Contin. Dyn. Syst. {\bf 35} (2015), no.~9, 4593--4610.
{\tt arXiv:1501.04873}

\bibitem{MyArt03}
S. P. S. Santos, N. Martins\ and\ D. F. M. Torres,
{\it An optimal control approach to Herglotz variational problems},
in {\it Optimization in the Natural Sciences}
(eds. A.~Plakhov, T.~Tchemisova and A.~Freitas),
Communications in Computer and Information Science,
Vol. 499, Springer, 2015, 107--117.
{\tt arXiv:1412.0433}

\bibitem{MyArt04}
S. P. S. Santos, N. Martins\ and\ D. F. M. Torres,
{\it Noether's theorem for higher-order variational problems of Herglotz type},
10th AIMS Conference on Dynamical Systems,
Dynamical Systems, Differential Equations and Applications,
Vol. 2015, AIMS Proceedings, 2015, 990--999.
{\tt arXiv:1507.05911}

\bibitem{Torres2001}
D. F. M. Torres,
{\it Conservation laws in optimal control},
in {\it Dynamics, bifurcations, and control} (Kloster Irsee, 2001), 287--296,
Lecture Notes in Control and Inform. Sci., 273, Springer, Berlin, 2002.

\bibitem{Torres2002}
D. F. M. Torres,
{\it On the Noether theorem for optimal control},
European Journal of Control {\bf 8} (2002), no.~1, 56--63.

\bibitem{MR1980565}
D. F. M. Torres,
{\it Gauge symmetries and Noether currents in optimal control},
Appl. Math. E-Notes {\bf 3} (2003), 49--57.
{\tt arXiv:math/0301116}

\bibitem{Torres2004}
D. F. M. Torres,
{\it Quasi-invariant optimal control problems},
Port. Math. (N.S.) {\bf 61} (2004), no.~1, 97--114.
{\tt arXiv:math/0302264}

\bibitem{MR2098297}
D. F. M. Torres,
{\it Proper extensions of Noether's symmetry theorem for nonsmooth extremals
of the calculus of variations},
Commun. Pure Appl. Anal. {\bf 3} (2004), no.~3, 491--500.

\end{thebibliography}
\end{document}